\documentclass[a4paper,11pt,reqno]{amsart}%
\usepackage{hyperref}
\usepackage{pdfsync}
\usepackage{dsfont}
\usepackage{color}
\usepackage{esint}
\usepackage{amsthm}
\usepackage{enumerate}
\usepackage{amsfonts}
\usepackage{amssymb}%
\usepackage{mathtools}
\usepackage{braket}
\usepackage{bm}
\usepackage{amsmath}
\usepackage{graphicx}
\usepackage{caption}
\usepackage{fullpage}
\numberwithin{equation}{section}
\newtheorem{theorem}{Theorem}[section]

\newtheorem{assumption}[theorem]{Assumption}

\newtheorem{definition}[theorem]{Definition}
\newtheorem{lemma}[theorem]{Lemma}

\newtheorem{proposition}[theorem]{Proposition}

\theoremstyle{remark}
\newtheorem{remark}[theorem]{Remark}

\newcommand{\argmin}{\mbox{argmin}}

\newcommand{\R}{\mathds{R}}
\newcommand{\N}{\mathds{N}}
\newcommand{\I}{\mathcal{I}}

\newcommand{\G}{\mathcal{G}}
\newcommand{\E}{\mathbb{E}}
\renewcommand{\O}{\mathcal{O}}
\newcommand{\F}{\mathcal{F}}
\newcommand{\h}{\mathcal{H}}
\newcommand{\X}{\mathcal{X}}

\renewcommand{\P}{\mathcal{P}}
\newcommand{\converges}[1]{ \overset{#1}{\longrightarrow}} 
\newcommand{\M}{\mathcal{M}}

\newcommand{\x}{\mathbf{x}}

\newcommand{\veps}{\varepsilon}
\newcommand{\dkl}{D_{\mbox {\tiny{\rm KL}}}}

\DeclareMathOperator{\esssup}{esssup}

\definecolor{mygreen}{rgb}{0.1,0.75,0.2}

\newcommand{\nc}{\normalcolor}

\title{Continuum Limits of  Posteriors in Graph Bayesian Inverse Problems}
\author{Nicol\'as Garc\'ia Trillos and Daniel Sanz-Alonso} 
\address{Division of Applied Mathematics, Brown University, Providence, RI, 02912, USA.}

\begin{document}
\maketitle

\begin{abstract}
We consider the problem of recovering a function input of a differential equation formulated on an unknown domain $\M$. We assume to have access to a discrete domain $\M_n=\{\x_1, \dots, \x_n\} \subset \M$, and to noisy measurements of the output solution at $p\le n$ of those points.
We introduce a graph-based Bayesian inverse problem, and show that the graph-posterior measures over functions in $\M_n$ converge, in the large $n$ limit, to a posterior over functions in $\M$ that solves a Bayesian inverse problem with known domain.
 The proofs rely on the variational formulation of the Bayesian update, and on a new topology for the study of convergence of measures over functions on point clouds to a measure over functions on the continuum. Our framework, techniques, and results may serve to lay the foundations of robust uncertainty quantification of graph-based tasks in machine learning. The ideas are presented in the concrete setting of recovering the initial condition of the heat equation on an unknown manifold.  

\end{abstract}

\section{Introduction}\label{sec:introduction}
\subsection{Aim and Relevance}
The Bayesian formulation of inverse problems has found numerous applications, notably in learning function inputs for differential equations \cite{AS10}, \cite{dashti2011uncertainty}, \cite{cotter2009bayesian}, \cite{biegler2011large}, \cite{trillos2016bayesian}.  A common, but often unnatural assumption is that the physical domain where the equation is defined is known. In this paper we propose a graph-based Bayesian approach for learning a function input of a differential equation on an unknown domain $\M$. We assume to have only access to a point cloud $\M_n=\{\x_1, \dots, \x_n\} \subset \M$, and to noisy measurements of the output solution at $p\le n$ of those points. The graph-posterior solutions to the discrete Bayesian inverse problems in the point cloud are proved to be close ---in the large $n$ limit--- to a continuum posterior that solves a Bayesian inverse problem on known domain $\M$.  

Our contributions build on, and extend, the growing field of continuum limits of functionals on point clouds \cite{trillos}, \cite{trillosACHA}. The main result, Theorem \ref{main:thm}, shows the convergence of graph-posteriors to the continuum posterior. The proof is based on $\Gamma$-convergence, and exploits two main ideas. First, that the Bayesian update admits a variational formulation \cite{trillossanzflows}: the posterior is characterized as  the minimizer of a functional acting on measures. Second, the introduction of a topology that allows to compare measures over functions on point clouds with measures over functions on the continuum. This topology is naturally defined in terms of the transportation metric introduced in \cite{trillos} for the study of the continuum limit of the total variation functional on point clouds.

Beyond setting forth the formulation of inverse problems on unknown domains, this paper aims to contribute to the theory of \emph{robust} uncertainty quantification of graph-based learning tasks such as clustering, classification, and semi-supervised learning. Several widely used algorithms were traditionally set as an optimization problem over functions, and the paper \cite{trillos} provides a natural framework for studying their continuum limits. These procedures have recently been subject to uncertainty quantification by casting them into a Bayesian formulation \cite{bertozziuncertainty}.  Here we study continuum limits at the level of measures by viewing the Bayesian update as an optimization problem over measures. Thereby we provide the tools for the study of robust uncertainty quantification of graph-based tasks in machine learning and statistics.

The presentation is framed in the concrete setting of the Bayesian inversion of the heat equation. To lay the ground for a more general analysis, we highlight the three building blocks required for the study of continuum limits of posteriors: continuum limit of priors, stability of forward maps, and stability of observation maps. These are established in Theorem \ref{lemma} in our concrete setting. Whereas the continuum limit of priors is readily transferable to other problems, the stability of forward and observation maps will have to be studied case by case. Our analysis makes use of recently developed spectral theory of graph Laplacians over point clouds in manifolds \cite{SpecRatesTrillos}. A more general study of stability of forward and observation maps will require  ---and serve as motivation for--- further study of regularity theory on geometric graphs. This is a much active research area, but still in its infancy.

\subsection{Outline}\label{ssec:outline}
Our setting is described in Section \ref{sec:contandgraphBIP}, where the (continuum) heat Bayesian inverse problem is reviewed, and its graph counterpart is introduced. 
We describe the variational formulation of these problems in Section \ref{sec:preliminaries}. We further introduce a natural topology for studying the convergence of measures over functions in the discrete domain to measures over functions in the continuous domain. In Section \ref{sec:mainresults} we state our main results on the continuum limit of graph-posteriors. Finally, Section \ref{sec:proofs} contains some auxiliary results and the proofs of our main results.

\subsection{Notation}\label{sec:notation}
Throughout, $\M$ denotes an $m$-dimensional, compact, smooth manifold embedded in $\R^d$; we refer to $\M$ as the \textit{continuous domain}; $d_\M$ denotes the geodesic distance in $\M$. We denote by $\gamma$ the uniform probability distribution on $\M$. We let  $\M_n:=\{\x_1, \dots, \x_n \}$ be a collection of i.i.d. samples from $\gamma$ and refer to $\M_n$ as the \textit{discrete domain}, and say that $\M_n$ is a point cloud. We denote by $\gamma_n$ the empirical measure
\[ \gamma_n := \frac{1}{n} \sum_{i=1}^n \delta_{\x_i}.\]
We write $L^2(\gamma)$ and $L^2(\gamma_n)$ to denote, respectively, $L^2(\M, \gamma)$ and $L^2(\M_n, \gamma_n).$ 
Any object (e.g. functions, maps, measures) built in the discrete domain will be denoted with a subscript $n$. We use boldface to denote measures over functions on $\M$ or $\M_n.$

\section{Continuum and Graph Bayesian Inverse Problems}\label{sec:contandgraphBIP}

We consider the inverse problem of recovering the initial condition of the heat equation from noisy observation of the solution at a positive time, which we choose to be $t=1$. The classical continuous setting is reviewed in Subsection \ref{ssec:continuous}, and the graph-based discrete setting is introduced in Subsection \ref{ssec:discrete}.

\subsection{Continuum Bayesian Inverse Problem}\label{ssec:continuous}
We first recall the classical setup where the physical domain $\M$ is known. The Bayesian formulation of inverse problems \cite{AS10} provides a principled way to learn an unknown input (here the initial condition $u$ of the heat equation), based on the three following ingredients:
\begin{enumerate}
\item A {\bf{prior}} distribution $\boldsymbol{\pi}$ on the unknown input $u$. The prior encodes all the available information on the unknown before data has been observed. Here we will assume a Gaussian prior
\begin{equation}\label{eq:prior}
\boldsymbol{\pi} = N\bigl(0,(\alpha I - \Delta)^{-s/2}\bigr),
\end{equation}
where $\alpha\ge 0$ is given, and $s>m$ so that $u\sim \boldsymbol{\pi}$ belongs to $L^2(\gamma)$ almost surely. See Remark \ref{RemL2L8} below for a discussion of our function space setup.
\item A {\bf{forward model}}. We consider the heat equation:
\begin{align}\label{eq:heat}
\begin{split}
  \begin{cases} v_t(x,t) = \Delta v(x,t) , \quad &t >0,\, x\in \M,  \\ v(x,0)= u(x), & x\in \M;   \end{cases}
  \end{split}
\end{align}\label{heat}
Let $\F:L^2(\gamma) \to L^2(\gamma)$ map an input $u\in L^2(\gamma)$ to the solution to \eqref{eq:heat} at time $t=1,$ $v(\cdot,1)\in L^2(\gamma).$ We call $\F$ the \emph{forward map}.
\item {\bf Data} $y$ consisting of some noisily measured output of the forward model. We assume that
\begin{equation}\label{eq:data}
y=[y_1, \ldots, y_p]^T= \bigl[\bar{v}^\delta(\x_1,1), \ldots, \bar{v}^\delta(\x_p,1) \bigr]^T + \eta,  \quad \eta(d\zeta) \propto \exp(-\phi(\zeta)) d\zeta,
\end{equation}
where $\delta>0$ is small and given, $\x_1, \ldots.
,\x_p\in \M$,
$$\bar{v}^\delta(\x_j,1) := \int_{B_\delta( \x_j) \cap \M} v(x,1) \gamma(dx), \quad 1\le j \le p ,$$
is the average of the function $v(\cdot,1)$ over an euclidean ball of radius $\delta>0,$
and $\eta$ represents observation noise with negative log-density $\phi$ with respect to Lebesgue measure in $\R^p.$ For instance, Gaussian observation noise corresponds to quadratic $\phi$. Throughout, $\phi$ is assumed to be continuous and bounded from below.

The map $$\O: L^2(\gamma) \to \R^p \quad \quad  v\in L^2(\gamma) \mapsto \bigl[\bar{v}^\delta(\x_1,1), \ldots, \bar{v}^\delta(\x_p,1) \bigr]^T$$ is called the \emph{observation map}. Denoting $\G := \O \circ \F,$ \eqref{eq:data} can be rewritten as
$$y = \G(u) + \eta. $$
From this expression one readily sees that the negative log-likelihood, representing the distribution of the data conditioned on the unknown $u,$ is given (with a slight abuse of notation) by 
\begin{equation}\label{def:phi}
\phi(u;y) := \phi\bigl(y-\G(u)\bigr).
\end{equation}
\end{enumerate}

The solution to the continuum Bayesian inverse problem is a continuum \emph{posterior} measure $\boldsymbol{\mu} \in \P(L^2(\gamma))$, 
which represents the conditional distribution of $u$ given the data $y.$ It can be represented as a change of measure from the prior
\begin{equation}\label{def:posterior}
\boldsymbol{\mu}(du) \propto \exp\bigl(-\phi(u;y)\bigr) \boldsymbol{\pi}(du),
\end{equation}
where it is tacitly assumed that $ \exp\bigl(-\phi(u;y)\bigr)$ is $\boldsymbol{\pi}$-integrable so that the right-hand side in \eqref{def:posterior} can be normalized into the probability measure $\boldsymbol{\mu}.$

The posterior $\boldsymbol{\mu}$ contains ---in a precise sense \cite{duke}--- all the information on the unknown input $u$ contained in the prior and the data. The measure $\F_\sharp \boldsymbol{\mu}$ contains all the available knowledge and remaining uncertainty on the output solution of the heat equation at time $t=1.$

\subsection{Graph Bayesian Inverse Problem}\label{ssec:discrete}
We now consider the setting of interest in this paper, where the domain $\M$ is unknown. We assume to have access, however, to $n$ uniform samples from $\M$, constituting a point cloud $\M_n:= \{\x_1, \dots, \x_n \}\subset \M.$ We further assume to have access to the same data $y$  as in the continuous setting, where the points $\x_1, \ldots, \x_p$ in \eqref{eq:data} are assumed to correspond to the first $p$ points in $\M_n$. The aim is to construct ---without knowledge of $\M$---  surrogate ``posterior'' measures $\boldsymbol{\mu_n}$ that can be proved to be close to $\boldsymbol{\mu}$ in the large $n$ limit. \nc

In order to carry out this program we endow  $\M_n$ with a graph structure. First, consider the kernel function $K : [0,\infty) \rightarrow [0, \infty)$ defined by
\begin{equation}
K(r):=\begin{cases}1 & \text{ if } r \leq 1, \\ 0 & \text{otherwise.} \end{cases} 
\label{eta}
\end{equation}
%and consider the quantity
%\begin{equation}
% \sigma_\eta:= \int_{0}^\infty  \eta(r) r^{d+1} dr = \frac{1}{d+2} .
% \label{sigmaeta}
%\end{equation}
Let
\[ \sigma_K := \int_{\R^m} K(|x|)\lvert \langle x , e_1 \rangle \rvert^2 d x = \frac{\alpha_m}{m+2},  \]
where $\alpha_m$ is the volume of the $m$-dimensional unit ball.
For $\veps>0$ we let $K_\veps: [0,\infty) \rightarrow [0, \infty)$ be the rescaled version of $K$ given by
\[ K_\veps(r) := \frac{1}{ \sigma_K n^2 \veps^{ m+2 }}K\left( \frac{r}{\veps} \right).\]
We define the weight $W_{n}(\x_i, \x_j)$ between $\x_i, \x_j \in \M_n$ by
\[ W_n(\x_i,\x_j) := K_{\veps_n}(\lvert  \x_i-\x_j \rvert), \]
for a given choice of parameter $\veps=\veps_n$, where we have made the dependence of the \emph{connectivity rate} $\veps$ on $n$ explicit. In order for our continuum limit results to hold, the connectivity rate needs to decrease with growing $n$ as stated in Assumption \ref{assumption} below.

After specifying the geometric graph $(\M_n, W_n)$, we introduce its  \textit{graph Laplacian}, which in matrix form can be written as
\[ \Delta_n := D_n -W_n,   \]
where $D$ is the degree matrix of the weighted graph, i.e.,  the diagonal matrix with diagonal entries $D_{ii}= \sum_{j=1}^n W_{n}(\x_i, \x_j)$. The above corresponds to the \textit{unnormalized} graph Laplacian, but different normalizations are possible \cite{von2007tutorial}.

We are ready to describe the graph Bayesian inverse problem. Each of the three ingredients (prior, forward map, and observation map) in the continuous setting is replaced by a discrete surrogate:
\begin{enumerate}
\item A {\bf prior} distribution $\boldsymbol{\pi_n}$ on an unknown input $u_n$. We assume
\begin{equation}\label{priordiscrete}
\boldsymbol{\pi_n} = N\bigl(0, (\alpha I_n +  \Delta_n)^{-s}\bigr),
\end{equation}
where $\alpha \ge 0$ and $s>m$ are as in \eqref{eq:prior}. The graph Laplacian, contrary to the regular Laplacian, is positive semi-definite, and hence the change in sign with respect to \eqref{eq:prior}. 
\item A {\bf forward model} 
\begin{align}\label{heatgraph}
\begin{split}
  \begin{cases} D_t v_n = -\Delta_n v_n , \quad t >0,  \\ v_n(0)= u_n.   \end{cases} 
\end{split}
\end{align}
The map $\F_n: L^2(\gamma_n) \to L^2(\gamma_n)$  that maps $u_n\in L^2(\gamma_n)$ to the solution $v_n(1)$ to \eqref{heatgraph} at time $t=1$ will be referred to as the graph forward map.

\item {\bf Data} For the construction of graph-posterior measures we use the data $y$ in \eqref{eq:data}, corresponding to noisily observing the value of the output solution at the first $p$ points of $\M_n$. We define a surrogate observation map
$$\O_n: L^2(\gamma_n) \to \R^p \quad \quad  v\in L^2(\gamma_n) \mapsto \bigl[\bar{v}_{n,1}, \ldots, \bar{v}_{n,p} \bigr]^T,$$
where $\delta >0$ is small and given,
$$\bar{v}_{n,j}^\delta :=\frac1n \sum_{k: \,\x_k\in B_\delta( \x_j) \cap \M_n} [v_n(1)]_k, \quad 1\le j \le p,$$ 
and $[v_n(1)]_k$ denotes the $k$-th entry of the vector $v_n.$
As before, we denote $\G_n :=\O_n \circ \F_n.$

%\begin{equation}\label{eq:datadisc}
%y=[y_1, \ldots, y_p]^T= \bigl[\bar{v}_{n,1}^\delta, \ldots, \bar{v}_{n,p}^\delta \bigr]^T + \eta, \quad \eta(d\zeta) \propto \exp(-\phi(\zeta)) d\zeta,
%\end{equation}
%where $\delta >0$ is small and given,
%$$\bar{v}_{n,j}^\delta =\frac1n \sum_{k: \,\x_k\in B_\delta( \x_j) \cap \M_n} [v_n(1)]_k, \quad 1\le j \le p,$$ 
%$[v_n(1)]_k$ denotes the $k$-th entry of the vector $v_n$, and $\eta$ represents observation noise with minus log-density $\phi$, as in the continuous setting.

% The map $$\O_n: L^2(\M_n) \to \R^p \quad \quad  v\in L^2(\M_n) \mapsto \bigl[\bar{v}_{n,1}, \ldots, \bar{v}_{n,p} \bigr]^T$$ is called the \emph{observation map}. 
\end{enumerate}
We define the \emph{graph-posterior} measure  $\boldsymbol{\mu_n} \in \P(L^2(\gamma_n))$ by
\begin{equation}\label{def:posteriordiscrete}
\boldsymbol{\mu_n}(du_n) \propto \exp\bigl(-\phi_n(u_n,y) \bigr) \boldsymbol{\pi_n} (du_n),
\end{equation}
where
\begin{equation}\label{def:phin}
\phi_n(u_n,y):= \phi\bigl(y-\G_n(u_n)\bigr).
\end{equation}
\nc

In this paper we show that ---provided that the graph connectivity parameter $\epsilon$ is chosen in terms of $n$ as in Assumption \ref{assumption}--- the graph-posteriors $\boldsymbol{\mu_n}$ in \eqref{def:posteriordiscrete} constitute good approximations in the large $n$ limit to the continuum posterior $\boldsymbol{\mu}$ in \eqref{def:posterior}.

\begin{remark}
Consider the graph-Bayesian inverse problem 
$$y_n = \G_n(u_n) + \eta,$$
where $u_n\sim \boldsymbol{\pi_n} $ and $\eta$ has negative log-density $\phi.$ Given data $y_n$, the posterior of this inverse problem is a measure $\boldsymbol{\mu_n^}{y_n}$ corresponding to the conditional distribution of $u_n$ given $y_n.$ The graph-posterior in \eqref{def:posteriordiscrete} is the measure $\boldsymbol{\mu_n} ^y$ obtained from pluggin in data $y$ obtained from the continuous model.
\end{remark}

\begin{remark}\label{remarketa}
Although for simplicity we focus on the kernel $K$ defined above, we anticipate that all our results are still true for more general kernels. In particular, the proofs of the main results of this paper are still valid if we simply assume that $K$ is a non-negative, compactly supported, and non-increasing function with $K(0)>0$. 
\end{remark}

\begin{remark}
\label{RemL2L8}
Our choice of inverse problem and function space setting is prompted by its simplicity and its power to illustrate the ideas. 
It also serves as motivation for the future study of finer analytical properties of the heat and Poisson equations on geometric graphs. These finer results would extend the regularity theory beyond the usual $L^2$ setting, and would potentially make possible to deal with pointwise observations rather than the averaged data considered here. These questions are the subject of current investigation.  
\end{remark}

\begin{remark}
In this paper we restrict our attention to the prior $\boldsymbol{\pi_n}$ and the forward map $\mathcal{F}_n$ defined using the full spectrum of the graph Laplacian $\Delta_n$. A close look to the proofs of our results suggests considering versions of graph priors and graph forward maps that ignore high frequencies, thus taking advantage of the reduced intrinsic dimension of the continuous Bayesian inverse problem \cite{agapiou2015importance} caused by the assumed regularity of the prior and the smoothing effect of the heat equation. Working with those modified priors and forward maps would simplify our proofs as in that case we would not have to consider what happens with the tails of the spectra. 
We would like to emphasize that ultimately $\boldsymbol{\pi_n}$ and $\mathcal{F}_n$ are user chosen and can be picked in any convenient way that guarantees the convergence towards the continuum limit. 
\end{remark}

\section{Variational Formulation and Topology}\label{sec:preliminaries}
In this section we introduce the two main ideas that will enable us to establish the continuum limits of graph-posteriors via $\Gamma$-convergence. First, in Subsection \ref{ssec:variationalformulation} we formulate the Bayesian update variationally, characterizing the posteriors as minimizers of certain functionals over measures. Second, the topology in which the consistency will be studied is introduced in Subsection \ref{ssec:topology}. Finally, we provide a brief background on $\Gamma$-convergence in Subsection \ref{ssec:gamma}.

\subsection{The Variational Formulation}\label{ssec:variationalformulation}
Recall that the Kullback-Leibler divergence (relative entropy) between two probability measures $\mathbb{P},\mathbb{Q}$ defined on a common measurable space $(\Omega, \mathfrak{F})$ is given by
\begin{equation*}\label{kldefinition}
 \dkl(\mathbb{P} \| \mathbb{Q}) := \begin{cases}  \int_{\Omega}  \frac{d \mathbb{P}}{d \mathbb{Q}}(u) \log\left( \frac{d\mathbb{P} }{d \mathbb{Q}}(u) \right)d \mathbb{Q}(u),  &  \mathbb{P} \ll \mathbb{Q},  \\ \infty, & \text{otherwise}.   \end{cases}  
 \end{equation*}

Then, as discussed in \cite{trillossanzflows}, the  posterior measure $\boldsymbol{\mu}$ in \eqref{def:posterior} can be characterized as the minimizer of the functional
\begin{equation}\label{def:j}
 J(\boldsymbol{\nu}) := \dkl(\boldsymbol{\nu} \| \boldsymbol{\pi}) + \int_{L^2(\gamma)}\phi\bigl(u ; y \bigr)\, d \boldsymbol{\nu} (u) , \quad \boldsymbol{\nu} \in \mathcal{P}\bigl(L^2(\gamma)\bigr),
\end{equation}
where, recall, $\boldsymbol{\pi}$ is the prior measure and $\phi(u;y)$ is the negative log-likelihood, given by \eqref{def:phi}.

Likewise, at the discrete level the graph-posterior $
\boldsymbol{\mu_n}$ can be characterized as the minimizer of 
\begin{equation}\label{def:jn}
  J_n(\boldsymbol{\nu_n}) := \dkl(\boldsymbol{\nu_n} \| \boldsymbol{\pi_n}) + \int_{L^2(\gamma_n)}\phi_n(u_n ; y)\, d \boldsymbol{\nu_n} (u_n) , \quad \boldsymbol{\nu_n} \in \mathcal{P}\bigl(L^2(\gamma_n)\bigr), 
\end{equation}
where $\boldsymbol{\pi_n}$ is the prior, and $\phi_n(u_n;y)$ is given by \eqref{def:phin}.

The variational characterization of posteriors
\begin{equation}\label{def:posteriorsvar}
\boldsymbol{\mu} = \argmin_{\boldsymbol{\nu} \in \P\bigl(L^2(\gamma)\bigr)} J(\boldsymbol{\nu}), \quad \boldsymbol{\mu_n} = \argmin_{\boldsymbol{\nu_n} \in \P(L^2\bigl(\M_n)\bigr)} J_n(\boldsymbol{\nu_n}),
\end{equation}
will enable us to study the convergence of $\boldsymbol{\mu_n}$ to $\boldsymbol{\mu}$ by studying the $\Gamma$-convergence of the functionals $J_n$ to $J.$ A brief review on $\Gamma$-convergence is given in Subsection \ref{ssec:gamma}. A key ingredient in the study of $\Gamma$-convergence is the choice of a suitable topology. We introduce such topology in the next subsection. We close the current one a remark.

\begin{remark} 
Since the function $t \in [0, \infty) \mapsto t \log(t) $ is strictly convex, the minimizers of $J_n$ and $J$ are unique.
\label{remarkuniqueness}
\end{remark}

%\begin{remark}
%Optimization over functions: MAP. Optimization over measures: Bayes. \red TO DO \nc
%\end{remark}

\subsection{Topology}\label{ssec:topology}
The paper \cite{trillos} introduced a natural topology for the study of the continuum limit of the graph total variation (a functional acting on functions on a point cloud) to the total variation functional (a functional acting on functions in the continuum). Building on this topology ---that we recall in Subsection \ref{sssec:tl2}--- we introduce in \ref{sssec:ptl2} a notion of convergence of \emph{measures} over functions in the point cloud to  \emph{measures} over functions in the continuum.
This new topology will be used to establish the $\Gamma$- convergence of the functionals $J_n$ in \ref{def:jn} (functionals acting on measures over functions on a point cloud) to the functional $J$ in \ref{def:j} (a functional acting on measures over functions in the continuum). 

\subsubsection{The Space $TL^2$}\label{sssec:tl2}
Consider the set 
\[ TL^2 := \bigl\{ (\theta, f) \; : \:  \theta \in \mathcal P(\M), \, f \in L^p(\M, \theta) \bigr\}. \]
For $(\theta_1,f_1)$ and $(\theta_2,f_2)$ in $TL^p$ we define, following \cite{trillos},  
\begin{align} \label{tlpmetric}
\begin{split}
 d_{TL^2}\bigl((\theta_1,f_1), (\theta_2,f_2)\bigr):=
   \inf_{\omega \in \Gamma(\theta_1, \theta_2)} \left(  \iint_{\M \times \M} \Bigl( d_\M(x,y)^2 + |f_1(x)-f_2(y)|^2  \Bigr) d\omega(x,y)  \right)^{1/2},
\end{split}
\end{align}
where $\Gamma(\theta_1,\theta_2)$ is the set of all Borel probability measures on $\M\times \M$ with marginal $\theta_1$ on the first factor and $\theta_2$ on the second one. It was shown in \cite{trillos} that $d_{TL^2}$ defines a distance in $TL^2$.

\begin{remark}
\label{CompletionTL2}
As noted in \cite{trillos},  $(TL^2, d_{TL^2})$ is not a complete metric space. Its completion, denoted $\overline{TL^2}$, can be identified to be the space $\mathcal{P}_2(\M \times \R)$ of Borel probability measures on the product space $\M \times \R$ with finite second moments, endowed with the Wasserstein distance. The space $\overline{TL^2}$ is a Polish space.  
\end{remark}

We will say, with a slight abuse of notation, that a sequence $u_n \in L^2(\gamma_n)$ converges in $TL^2$ to $u\in L^2(\gamma),$ written
$$u_n \converges{TL^2} u,$$
 if $d_{TL^2}\bigl((\gamma_n,u_n), (\gamma,u)\bigr) \to 0.$  A characterization of convergence in $TL^2$ in terms of composition with transport maps can be found in Proposition 3.12 in \cite{trillos}.

\subsubsection{The Space $\P(TL^2)$}\label{sssec:ptl2}
Let $\theta \in \mathcal{P}(\M)$ and  $\boldsymbol{\nu} \in \mathcal{P}(L^2(\theta))$. We note that $\boldsymbol{\nu}$  can be thought naturally as a measure in $\mathcal{P}(TL^2)$. Indeed, the canonical inclusion
\[ \mathcal{I}_\theta:  f \in  L^2(\theta) \longmapsto   (\theta, f) \in TL^2 \]
induces the canonical inclusion
\[ \mathcal{I}_{\theta \sharp}:  \mathcal{P}(L^2(\theta)) \hookrightarrow \mathcal{P}(TL^2), \]
where $\mathcal{I}_{\theta \sharp}$ is the push-forward via $\mathcal{I}_{\theta}$.  Notice that $\mathcal{I}_\theta$ is a continuous map. In the sequel we may drop the explicit mention to $\mathcal{I}$ whenever no confusion rises from doing so.

The above observation motivates the following definition. 
\begin{definition}
\label{convmeasures:defn}
For  $\boldsymbol{\nu_n} \in \mathcal{P}(L^2(\gamma_n)),$ $n\in \N,$ and   $\boldsymbol{\nu} \in \mathcal{P}(L^2(\gamma))$ we say that $\{\boldsymbol{\nu_n}\}_{n\in \N}$ converges to $\boldsymbol{\nu}$, written
 $$\boldsymbol{\nu_n}\converges{\P(TL^2)}{ \boldsymbol{\nu}},$$
  if $\{I_{\gamma _n \sharp} \boldsymbol{\nu_n}\}_{n\in \N}$ converges weakly to $ I_{\gamma \sharp}  \boldsymbol{\nu}$ in $\mathcal{P}(TL^2).$
\end{definition}

The following result is analogous to Lemma 2.7 in \cite{trillosACHA}.

\begin{proposition}(Stability of composition)
Suppose that  
\begin{enumerate}
\item $\boldsymbol{\nu_n}\converges{\P(TL^2)}{ \boldsymbol{\nu}}$.
\item The maps $\mathcal{H}_n : L^2(\gamma_n) \rightarrow L^2(\gamma_n)$ converge in $TL^2$ towards $\mathcal{H} : L^2(\gamma) \rightarrow L^2(\gamma)$.
\end{enumerate}
Then, 
\[   \mathcal{H}_{n \sharp} \boldsymbol{\nu_n} \converges{\P(TL^2)}{ \mathcal{H}_{\sharp} \boldsymbol{\nu}.}   \]
\label{compos:prop}
\begin{proof}
By Skorohod's theorem there exists a probability space $(\widetilde{\Omega}, \widetilde{\mathfrak{F}}, \widetilde{\mathbb{P}})$ where one can define $TL^2$-valued random variables $\{X_n\}_{n\in\N}$ and $X$ with distributions  $\{\boldsymbol{\nu_n}\}_{n\in\N}$ and $\boldsymbol{\nu}$, respectively,  such that, for $\tilde{\mathbb{P}}$ almost all $\tilde{\omega},$ $$X_n(\tilde{\omega}) \to X(\tilde{\omega}).$$
The assumed convergence of $\h_n$ to $\h$ implies that, for $\tilde{\mathbb{P}}$ almost all $\tilde{\omega},$
$$\h_n\bigl(X_n(\tilde{\omega})\bigr) \to \h \bigl( X(\tilde{\omega})\bigr).$$
Since $\h_n(X_n)\sim \h_{n\sharp}\boldsymbol{\nu_n}$ and $\h(X) \sim \h_\sharp \boldsymbol{\nu}$ the desired convergence follows.
\end{proof}
\end{proposition}

We finish this subsection with a remark.

\begin{remark}
Since $TL^2$ is dense in $\overline{TL^2}$ it follows that the weak convergence in $\mathcal{P}(TL^2)$ generates the same topology as the restriction of weak convergence in $\mathcal{P}(\overline{TL^2})$ to $\mathcal{P}(TL^2)$. This simple observation will be useful in the sequel; for example, we may obtain pre-compactness of sequences in $\mathcal{P}(TL^2)$ by first using Prokhorov's theorem to obtain pre-compactness in $\mathcal{P}(\overline{TL^2})$ and then showing that cluster points necessarily are in $\mathcal{P}(TL^2)$.  
\end{remark}

\subsection{Gamma Convergence}\label{ssec:gamma}
In this subsection we recall the notion of $\Gamma$-convergence in the general setting of a metric space $(\X, d_\X)$. 

\begin{definition}
\label{GamaConv:defn}
We say that a sequence of functionals $E_n: \X\rightarrow [0,\infty]$ $\Gamma$-converges to the functional $E: \X \rightarrow [0,\infty]$ if the following three properties are satisfied:
\begin{enumerate}
\item Liminf inequality: For every $x\in \X$ and every sequence $\{x_n\}_{n\in \N}$ converging to $x$,
$$\liminf_{n\to\infty} E_n(x_n) \ge E(x).$$
\item Limsup inequality: For every $x\in \X$ there exists a sequence $\{x_n\}_{n\in \N}$ converging to $x$ satisfying 
$$\limsup_{n\to \infty} E_n(x_n) \le E(x).$$
\item Compactness property: Given $\{n_k\}_{k\in \N}$ an increasing sequence of natural numbers and $\{x_k\}_{k\in \N}$ a bounded sequence in $\X$ for which 
$$\sup_{k\in \N} E_{n_k} (x_k) <\infty$$
$\{x_k\}_{k\in \N}$ is relatively compact in $\X.$
\end{enumerate}
\end{definition}

\begin{remark}
\label{DenseSet}
We say that a sequence $\{ x_n \}_{\in \N}$ is a recovery sequence for $x$ if it satisfies the conditions in the definition of the limsup inequality. We remark that in order to obtain a recovery sequence for all $x \in \X$ it is enough to obtain a recovery sequence for all elements in a subset $\tilde{\X}\subset \X$ satisfying the following property: For every $x \in \X$ there exists a sequence $\{ x_k \}_{k \in \N} \subseteq \tilde{\X}$ converging towards $x$ and for which $E(x_k) \to E(x)$. Such set $\tilde{\X}$ is said to be dense in $\X$ with respect to $E$. See \cite{dal2012introduction}.
\end{remark}

Traditionally, only the first two properties are included in the definition of $\Gamma$-convergence, the third one usually referred to as compactness property or coercivity. With our definition of $\Gamma$-convergence we can succinctly state the following.

\begin{proposition} (Stability of minimizers) 
Let $\{E_n\}_{n\in \N}$ be a sequence of functionals that are not identically equal to $\infty$ and $\Gamma$-converges towards the functional $E$, which is not identically equal to $\infty.$ Then
$$\lim_{n\to\infty} \inf_{x\in \X} E_n(x) = \min_{x\in \X} E(x).$$
Furthermore every bounded sequence $\{x_n\}_{n\in \N}$ in $\X$ for which 
$$\lim_{n\to\infty}\Bigl( E_n(x_n) - \inf_{x\in \X} E_n(x) \Bigr) = 0$$
is relatively compact and each of its cluster points is a minimizer of $E.$ 
In particular, if $E$ has a unique minimizer, then a sequence $\{x_n\}_{n\in \N}$ satisfying converges to the unique minimizer of $E.$
\label{GammaConv}
\end{proposition}

\begin{remark}
In this paper we use the notion of $\Gamma$-convergence in the context of $\mathcal{X}=\P(TL^2)$ \nc with the topology of weak convergence of probability measures. Notice that the Levy-Prokhorov metric metrizes the weak convergence of probability measures. 
\end{remark}

\section{Main Results}\label{sec:mainresults}
%We assume that $\gamma_n$ approximates the measure $\gamma$ (the uniform measure on $\M$) in a sense that we now make precise. Given two arbitrary probability measures $\mu, \tilde{\mu} \in \mathcal{P}(\M)$, their $\infty$-OT distance, denoted $d_\infty(\mu, \tilde{\mu})$, is defined as
%\[ d_\infty(\mu, \tilde{\mu}):= \min_{\pi \in \Gamma(\mu, \tilde{\mu})} \sup_{(x,y) \in \sup(\pi)} d_\M( x,y). \]

We will make the following assumption on the parameter $s$ that characterizes the regularity of draws from the prior $\boldsymbol{\pi}$ and  on the connectivity rate $\veps_n$ of the graph $(\M_n, W_n).$ Recall that $m$ denotes the dimension of the manifold $\M.$

\begin{assumption}\label{assumption}
It holds that $s>2m$. The connectivity rate $\veps_n$ satisfies
\[ \frac{(\log(n))^{p_m} }{n^{1/m}}   \ll \veps_n \ll \frac{1}{n^{1/s}}  , \quad \text{ as } n \rightarrow \infty,\]
where  $p_m =3/4 $ for $m =2$ and $p_m= 1/m$ for $m \geq 3$. 
\end{assumption}
The lower bound also appears in \cite{trillos} and \cite{trillosACHA}. It is related to the $\infty$-OT distance between $\gamma_n$ and $\gamma$. Indeed, with probability one, one can find transport maps $T_n: \M \rightarrow \M_n$ with $T_{n \sharp}\gamma = \gamma_n $ such that
\begin{equation}
 t_n := \esssup_{x \in \M} d_{\M}(x , T_n(x))   
\label{Tn}
\end{equation}
scales like $\frac{(\log(n))^{p_m}}{n^{1/m}}$ as $n \rightarrow \infty$. This is the best scaling achievable (see \cite{SpecRatesTrillos} for the result on manifolds or \cite{trillos2014canadian} for the result on Euclidean domains).   On the other hand, the upper bound on $\veps_n$ is assumed so as to guarantee that higher frequencies of $\Delta_n$ do not accumulate in the large $n$ limit. We notice that the condition $s>m$ is sufficient for defining the Gaussian measure $\boldsymbol{\pi}$ on $\mathcal{P}(L^2(\gamma))$. However, our proof of consistency of priors in Theorem \ref{lemma} below requires the stronger condition $s>2m.$

Our first main result establishes the continuum limit of graph Gaussian priors \ref{priordiscrete} to the continuum Gaussian prior  in the continuum \ref{eq:prior}. Moreover it shows the stability of the forward and observation maps defined in Section \ref{sec:contandgraphBIP}.

%We point out that the $\infty$-OT distance between $\nu$ and $\nu_n$ admits a formulation in terms of transportation \textit{maps}
%\[ d_\infty(\nu, \nu_n)= \min_{T_{n \sharp} \nu = \nu_n }  \esssup_{ x \in \M} d_\M(x, T_n(x))   ,\]
%where the $\min$ in the above formula is taken over all transportation maps between $\nu$ and $\nu_n$. The existence and uniqueness of minimizers to this problem is established in \cite{Champion}, as well as the equivalence of $d_\infty(\nu, \nu_n)$ with the minimization problem involving transportation maps.

\begin{theorem}\label{lemma}
Suppose that Assumption \ref{assumption} holds. Let $\boldsymbol{\pi_n}$, $\boldsymbol{\pi}$, $\F_n$, $\F$, $\O_n$, and $\O$ be priors, forward maps, and observation maps, as defined in Section \ref{sec:contandgraphBIP}. 
Then,
\begin{enumerate}
\item (Continuum limit of priors.)  $\boldsymbol{\pi_n} \converges{P(TL^2)} \boldsymbol{\pi}.$
\item (Stability of forward maps.) If $u_n \converges{TL^2} u$, then $\F_n(u_n) \converges{TL^2} \F(u).$
\item (Stability of observations maps.) If $v_n \converges{TL^2}{v},$ then $\O_n(v_n) \converges{}{ \O(v)}.$
\end{enumerate}
\end{theorem}

\begin{remark}
The assumption $\veps_n \ll \frac{1}{n^{1/s}}$ in the above theorem is only needed for the stability of priors. That condition can be replaced with $\veps_n \ll 1$ for the stability of forward maps to hold.
\end{remark}

The second main result of this paper is the following. 

\begin{theorem}(Continuum limit of posteriors.)
\label{main:thm}
Under Assumption \ref{assumption} the sequence of graph-posteriors $\{ \boldsymbol{\mu_n} \}_{n \in \N}$ converges towards the  posterior $\boldsymbol{\mu}$ in $\P(TL^2).$

Moreover, 
\[   \mathcal{F}_{n\sharp } \boldsymbol{\mu_n} \converges{\P(TL^2)}{ \mathcal{F}_{\sharp } \boldsymbol{\mu}}.\] 
\end{theorem}

\begin{remark}
Besides providing a clear and systematic way of obtaining stability of the graph-posteriors $\boldsymbol{\mu_n}$, the notion of $\Gamma$-convergence is also an important part when establishing stability of gradient flows. This has been pointed out in \cite{serfaty2011gamma} and \cite{gigli2010heat}. The connection of gradient flows of energies such as $J_n$ in Wasserstein space with the Bayesian update, as well as the connection with sampling methodologies, has been made apparent in \cite{trillossanzflows}.
\end{remark}

%is naturally defined. 

%The push-forward of measures on $X$ via $\tilde{T}_n$ defines a map 
%\[  Q_n : \mathcal{P}(X) \rightarrow \mathcal{P}(X_n) \]
% \[  Q_n(\mu):= \tilde{T}_{n\sharp} \mu   \]
%It is straightforward to check that for any measure $\mu_n \in \mathcal{P}(X_n) $, it is always possible to find a measure $\mu \in \mathcal{P}(X)$  for which $Q_n()$  Indeed,...
%For any $\mu$ in the pre-image of $\mu_n$,  the change of variables formula gives
%\[   \int_{X_n}    \mathcal{G}_n(u_n) d \mu_n(u_n)  = \int_X \mathcal{G}_n (  \tilde{T}_n (u)  ) d \mu(u)       \]
%for any bounded continuous function $\mathcal{G}_n$.

\section{Proofs of Main Results}\label{sec:proofs}
In this section we prove our main results. Some auxiliary lemmas are collected in Subsection \ref{ssec:auxiliary}, and the main theorems are proved in Subsection \ref{ssec:mainproofs}.
\subsection{Auxiliary Results}\label{ssec:auxiliary}
\begin{lemma}
\label{lemmaaux}
Let $\mathcal{X}$ be a complete separable metric space and suppose that $\{ \nu_n \}_{n\in \N} $ and $\{   \theta_n \}_{n\in \N}$ are two sequences in $\mathcal{X}$ satisfying
\begin{enumerate}
\item $\sup_{n \in \N}  \dkl( \theta_n \| \nu_n ) < \infty .$
\item $ \nu_n \converges{w} \nu $, for some $
\nu \in \mathcal{P}(\mathcal{X})$.
\end{enumerate}
Then, $\{ \theta_n \}_{n \in \N}$ is weakly pre-compact. Moreover, any of its cluster points is absolutely continuous with respect to $\nu$. 
\end{lemma}
\begin{proof} 
Let $\epsilon>0$ and set 
\[  \rho_n:= \frac{d \theta_n}{d \nu_n}.\]
To prove the result we can not apply the standard  De la Vall\'{e}e--Poussin's theorem (see \cite{meyer1966probability}) given that $\nu_n$ changes with $n$. Nevertheless, using the super-linearity of the function $\Phi(t) = t \log t$, and the same idea of proof it is possible to show that assumption (1) implies that 
$$\lim_{c\to \infty} \sup_{n\in N} \int_{ |\rho_n|\ge c} |\rho_n| d\nu_n = 0.$$
and hence there is $c^*>0$ such that, for all $n\in \N,$
$$\int_{ |\rho_n|\ge c^*} |\rho_n| d\nu_n \le \epsilon/2.$$
Because of the assumed tightness of $\{\nu_n\}_{n\in \N}$ we can choose a compact set $K$ such that 
$$\nu_n(\X/K) \le \frac{\epsilon}{2c^*}.$$
Thus 
$$\theta_n(\X/K) = \int_{\X/K} \rho_n d\nu_n \le \frac{\epsilon}{2} + c^* \frac{\epsilon}{2c^*} = \epsilon,$$
and $\{\theta_n\}_{n\in\N}$ is tight, and hence pre-compact. 
For the last part, suppose without loss of generality that $\theta_n \converges{w} \theta.$ Then, using the assumption (2) and the lower semicontinuity of the Kullback-Leibler divergence \cite{dupuis2011weak}, 
$$ \dkl(\theta\| \nu) \le \liminf_{n\to \infty} \dkl(\theta_n\| \nu_n) <\infty.$$
This shows that $\theta$ is absolutely continuous with respect to $\nu$ and completes the proof.
 
\end{proof}

We make use of the following extension of Fatou's lemma, proved in \cite{fatou}.
\begin{lemma}
\label{lemmaaux1}
Let $(\mathcal{X},d_\X)$ be an arbitrary Polish space and suppose that $\{  H_n \}_{n \in \N}$ is a sequence of lower-semi-continuous functions $H_n : \mathcal{X} \rightarrow [0,\infty]$ bounded uniformly from below and satisfying
\[  \liminf_{n \rightarrow \infty }  H_n(x_n)  \geq H(x),    \]
whenever 
\[x_n \rightarrow x.\]
In the above, $H$ is another function $H : \mathcal{X} \rightarrow [0,\infty]$.  In addition, suppose that $\{ \nu_n \}_{n \in \N}$ is a sequence in $\mathcal{P}(\mathcal{X})$ converging weakly towards $\nu \in \mathcal{P}(\mathcal{X}) $. Then, 
\[ \liminf_{n \rightarrow \infty} \int_{\mathcal{X}} H_n(x) d \nu_n(x)  \geq \int_{\mathcal{X}} H(x) d \nu(x).  \]
\end{lemma}
\subsection{Proof of Main Results}\label{ssec:mainproofs}

\subsubsection{Proof of Theorem \ref{lemma}}

Throughout the proof we let $\{ \psi_1^n, \dots, \psi_n^n \}$ be an orthonormal basis (in $L^2(\gamma_n)$)  of eigenfunctions of $\Delta_n$ with associated eigenvalues
\[\lambda_1^n \leq \dots \leq \lambda_n^n.\]
Likewise, we let $\{ \psi_1, \dots, \psi_n, \dots \}$ be an orthonormal basis (in $L^2(\gamma)$) of eigenfunctions of $-\Delta$ with associated eigenvalues
\[\lambda_1\leq \dots \leq \lambda_n \leq \dots.\]
By the results in \cite{trillosACHA} we can assume without the loss of generality that, for all fixed $l \in \N,$
\begin{equation}
\psi_l^n \converges{TL^2} \psi_l ,
\label{ConvEigenfunc}
\end{equation}
and that 
\begin{equation}
\lambda_{l}^n \rightarrow \lambda_l.
\label{ConvEigenvalue}
\end{equation}
For the purposes of establishing the consistency of prior distibutions (first part in Theorem \ref{lemma}) we need the finer results from \cite{burago2014graph}; the results are generalized in \cite{SpecRatesTrillos} to allow for non-uniform densities and different versions of graph Laplacians.  The results in \cite{burago2014graph} quantify the rate of convergence for a portion of the spectrum in terms of $t_n$ (defined in \eqref{Tn}) and $\veps_n$. In particular the results imply that there exists a constant $C>0$ such that, for all large enough $n$,
\begin{equation}
\label{ConvRates}
\lvert \lambda_i - \lambda_i^n \rvert \leq C \left( \frac{t_n}{\veps_n}  + \sqrt{\lambda_i} \veps_n \right) \leq C \left( \frac{t_n}{\veps_n} + k_n^{1/m}\veps_n \right), 
\end{equation}
for all $i=1, \dots, k_n$, where $k_n$ is chosen so that $k_n \leq c \lceil \frac{1}{\veps_n^m} \rceil $ for some small enough constant $c>0$.
The last inequality in \eqref{ConvRates} follows from Weyl's law for the growth of eigenvalues of $-\Delta$.   \nc

\begin{proof}[Continuum limit of priors]

Consider a probability space $(\widetilde{\Omega}, \widetilde{\mathfrak{F}}, \widetilde{\mathbb{P}})$ supporting i.i.d random variables $\{\xi_i \}_{i \in \N}$ with
\[ \xi_i \sim N(0,1). \]
For every $n \in \N$ we construct the $TL^2$-valued random variable
\[  X^n:= \sum_{i=1}^n (\alpha + \lambda_i^n)^{-s/4} \xi_i \, \psi_i^n, \]
where we notice that $X^n\sim \boldsymbol{\pi_n}$. Likewise, we construct the $TL^2$-valued random variable
\[ X:= \sum_{i=1}^\infty (\alpha + \lambda_i)^{-s/4} \xi_i \, \psi_i.  \]
We notice that $X$ is a well defined random variable with values in $L^2(\gamma)$ (and hence in $TL^2$) due to the assumption $s > m$. Moreover, $X \sim \boldsymbol{\pi}$.

The triangle inequality implies that
\begin{align}
\begin{split}
\label{auxPriors}
\widetilde{\E} \Bigl[d_{TLˆ2}\bigl(  (\gamma_n , X^n) , (\gamma, X)  \bigl) \Bigr]  \leq & \widetilde{\E}\Bigl[d_{TL^2}\bigl( ( \gamma_n , X^n) , ( \gamma_n , X^n_{k_n} )  \bigr) \Bigr]
\\ & + \widetilde{\E}\Bigl[d_{TL^2}\bigl( ( \gamma_n , X^n_{k_n}) , ( \gamma , X_{k_n} )  \bigr) \Bigr]
\\& + \widetilde{\E}\Bigl[d_{TL^2}\bigl( ( \gamma, X_{k_n}) , ( \gamma , X )  ) \Bigr],
\end{split}
\end{align}
where
\[ X^n_{k_n}:= \sum_{i=1}^{k_n} (\alpha+\lambda_i^n)^{-s/4}\xi_i \psi_i^n, \]
\[ X_{k_n}:= \sum_{i=1}^{k_n} (\alpha+\lambda_i)^{-s/4}\xi_i \psi_i, \]
for $k_n\le n$ chosen so that 
\[ n^{m/s} \ll k_n \ll \lceil \frac{1}{\veps_n^m} \rceil. \]
Notice that such $k_n$ can indeed be constructed due to the assumptions on $\veps_n$. We now show that each of the terms on the right hand side of \eqref{auxPriors} goes to zero as $n \rightarrow \infty$. 

First of all 
\[ \widetilde{\E}(d_{TL^2}( (\gamma,X_{k_n} ) , (\gamma , X)  ) ) \leq \widetilde{\E}(\lvert X- X_{k_n}  \rvert_{L^2(\gamma)}) \leq  \left( \sum_{i=k_n+1}^\infty (\alpha + \lambda_i)^{-s/2} \right)^{1/2}  , \] 
 where the last inequality follows after an application of Jensen's inequality for concave functions. It then follows that
\[ \lim_{n \rightarrow \infty} \widetilde{\E}\Bigl[d_{TL^2}\bigl( (\gamma,X_{k_n} ) , (\gamma, X )\bigr)\Bigr]=0. \]

Similarly, we see that
\[ \widetilde{\E}\Bigl[d_{TL^2}\bigl( (\gamma_n,X^n ) , (\gamma_n , X_{k_n}^n)   \bigl)\Bigr]\leq   \left( \sum_{i=k_n+1}^\infty (\alpha + \lambda_i)^{-s/2} \right)^{1/2} \leq  \left( \frac{n}{(\alpha + \lambda_{k_n}^n  )^{s/2}}  \right)^{1/2}. \] 
Using that $n^{m/s}\ll k_n$, \eqref{ConvRates} and Weyl's law we deduce that 
\[ \lim_{n \rightarrow \infty} \widetilde{\E}\Bigl[d_{TL^2}\bigl( (\gamma_n,X^n ) , (\gamma_n,X_{k_n}^n)\bigl)\Bigr]=0. \]

A further application of the triangle inequality gives
\[ \widetilde{\E}\Bigl[d_{TL^2}\bigl((\gamma_n, X_{k_n}^n) , ( \gamma, X_{k_n} ) \bigr)   \Bigr] \leq  \widetilde{\E}\Bigl[d_{TL^2}\bigl((\gamma_n, X_{k_n}^n) , ( \gamma_n, \hat{X}_{k_n}^n ) \bigr)   \Bigr] 
+ \widetilde{\E}\Bigl[d_{TL^2}\bigl((\gamma_n, \hat{X}_{k_n}^n) , ( \gamma, X_{k_n} ) \bigr)  \Bigr], \]
where
\[  \hat{X}_{k_n}^n := \sum_{i=1}^{k_n}(\alpha+ \lambda_i)^{-s/4}\xi_i \psi_i^n. \]
It then follows that
\begin{align*}
 \widetilde{\E}\Bigl[d_{TL^2}\bigl((\gamma_n, X_{k_n}^n) , ( \gamma, X_{k_n} ) \bigr)  \Bigr] \leq &\left( \sum_{i=1}^{k_n} \Bigl( (\alpha + \lambda_i^n)^{-s/4} - ( \alpha + \lambda_i)^{-s/4}  \Bigr)^2   \right)^{1/2}  \\
 &+ \sum_{i=1}^{k_n} (\alpha+ \lambda_i)^{-s/4} \lVert \psi_i - \psi_i^n\circ T_n \rVert_{L^2(\gamma_n)}, 
\end{align*}

where $T_n: \M \rightarrow \M_n$ is the transport map in 
 \eqref{Tn}. From $\eqref{ConvRates} $ it follows that
\begin{equation} \lim_{n \rightarrow \infty}  
\left( \sum_{i=1}^{k_n}\Bigl( (\alpha + \lambda_i^n)^{-s/4} - ( \alpha + \lambda_i)^{-s/4}  \Bigr)^2  \right)^{1/2} =0.
\label{AuxPriors2}
\end{equation}
On the other hand, for every fixed $l \in \N$ we have
\[ \sum_{i=1}^{k_n}  (\alpha+ \lambda_i)^{-s/4} \lVert \psi_i - \psi_i^n\circ T_n \rVert_{L^2(\gamma_n)} \leq \sum_{i=1}^{l} (\alpha+ \lambda_i)^{-s/4} \lVert \psi_i - \psi_i^n\circ T_n \rVert_{L^2(\gamma_n)} +  2 \sum_{i=l+1}^{k_n}(\alpha+ \lambda_i)^{-s/4}.  \] 
Combining the above with \eqref{AuxPriors2} and \eqref{ConvEigenfunc} we obtain
\[ \limsup_{n \rightarrow \infty} \widetilde{\E}\Bigl[d_{TL^2}\bigl((\gamma_n, X_{k_n}^n) , ( \gamma, X_{k_n}) \bigr)   \Bigr] \leq  \sum_{i=l+1}^{\infty}(\alpha+ \lambda_i)^{-s/4}, \quad \forall l \in \N. \]
Since $s>2m$, taking the limit of the right hand side of the above inequality as $l \rightarrow \infty$ gives
\[ \lim_{n \rightarrow \infty} \widetilde{\E}\Bigl[ d_{TL^2}\bigl( (\gamma_n, X_{k_n}^n) , (\gamma, X_{k_n})  \bigr) \Bigr]=0. \]

Putting together all the previous partial results we conclude that
\[ \lim_{n \rightarrow \infty} \widetilde{\E}\Bigl[ d_{TL^2}\bigl( (\gamma_n, X^n) , (\gamma, X)  \bigr) \Bigr]=0. \]
This implies that $\boldsymbol{\pi_n} \converges{\mathcal{P}(TL^2)} \boldsymbol{\pi}$.
\end{proof}

\begin{proof}[Consistency of forward maps]  
Suppose that $u_n \converges{TL^2}{u},$  so in  particular
\begin{equation}
C:= \sup_{n\in\N} \|u_n\|_{L^2(\gamma_n)} <\infty.
\end{equation} 

Note that $$\F_n(u_n) = \sum_{i=1}^n \exp(-\lambda_i^n) \langle u_n,\psi_i^n \rangle \psi_i^n,$$
$$\F(u) = \sum_{i=1}^n \exp(-\lambda_i) \langle u,\psi_i \rangle \psi_i.$$
For a fixed $l\in \N$ and $n>l$ we write
\begin{align*}
\F_n(u_n) &= \sum_{i=1}^l \exp(-\lambda_i^n) \langle u_n,\psi_i^n \rangle \psi_i^n   + \sum_{i=l+1}^n \exp(-\lambda_i^n) \langle u_n,\psi_i^n \rangle \psi_i^n \\
&=:\F_n(u_n;l)   + \sum_{i=l+1}^n \exp(-\lambda_i^n) \langle u_n,\psi_i^n \rangle \psi_i^n,
\end{align*}
and likewise
\begin{equation}
\F(u) =: \F(u;l)   + \sum_{i=l+1}^n \exp(-\lambda_i) \langle u,\psi_i \rangle \psi_i.
\end{equation}
The triangle inequality gives
\begin{align*}
 d_{TL^2}\Bigl(\bigl(\gamma_n, \F_n(u_n)\bigr), \bigl(\gamma, \F(u) \bigr)\Bigr) \le\,\, &  d_{TL^2}\Bigl(\bigl(\gamma_n, \F_n(u_n)\bigr), \bigl(\gamma, \F_n(u_n;l) \bigr)\Bigr) \\ &+ d_{TL^2}\Bigl(\bigl(\gamma_n, \F_n(u_n;l)\bigr), \bigl(\gamma, \F(u;l) \bigr)\Bigr) \\  &+
 d_{TL^2}\Bigl(\bigl(\gamma_n, \F(u;l)\bigr), \bigl(\gamma, \F(u) \bigr)\Bigr). 
\end{align*}
We bound each of the terms in the right hand-side.
For the first term,
\begin{align*}
d_{TL^2}\Bigl(\bigl(\gamma_n, \F_n(u_n)\bigr), \bigl(\gamma, \F_n(u_n;l) \bigr)\Bigr) &\le \| \F_n(u_n) - \F_n(u_n;l)\|_{L^2(\gamma_n)} \\
&\le \biggl(  \sum_{i=l+1}^n   \exp(-2 \lambda_i^n)  \langle u_n, \psi_i^n \rangle   \biggr)^{1/2} \\
& \le \exp(-\lambda_l^n) \| u_n \|_{ L^2(\gamma_n)}\\
& \le \exp(-\lambda_l^n)C.
\end{align*}
Similarly, for the third one,
\begin{align*}
d_{TL^2}\Bigl(\bigl(\gamma, \F(u)\bigr), \bigl(\gamma, \F(u;l) \bigr)\Bigr) &\le \| \F(u) - \F(u;l)\|_{L^2(\gamma)} \\
& \le \exp(-\lambda_l) \| u \|_{ L^2(\gamma)}.
\end{align*}
For the second one, using \cite{trillosACHA}
$$\F_n(u_n;l) \converges{TL^2}{\F(u)},$$
and therefore 
$$d_{TL^2}\Bigl(\bigl(\gamma_n, \F_n(u_n;l)\bigr), \bigl(\gamma, \F(u;l) \bigr)\Bigr) \to 0.$$
Thus, 
\begin{equation}
\label{auxContF}
 \limsup_{n \rightarrow \infty}  d_{TL^2}\Bigl(\bigl(\gamma_n, \F_n(u_n)\bigr), \bigl(\gamma, \F(u) \bigr)\Bigr) \leq  2C \exp(- \lambda_l)    ,   
 \end{equation}
where we have used that 
$  \lambda_{l}^n \to \lambda_l . $

We may now take the limit of the right hand side of \eqref{auxContF} as $l \to  \infty$ to deduce that
\[ \limsup_{n \rightarrow \infty} d_{TL^2}\Bigl(\bigl(\gamma_n, \F_n(u_n)\bigr), \bigl(\gamma, \F(u) \bigr)\Bigr) =0.  \]
Hence, $\mathcal{F}_n(u_n) \converges{TL^2} \mathcal{F}(u)$ as we wanted to prove. 
\end{proof}

\begin{proof}[Consistency of observation maps]

Suppose that $v_n \converges{TL^2} v $. We want to show that $\O_n(v_n) \rightarrow \O(v)$ in $\R^p$. Notice that the $i$-th coordinates of $\O_n(v_n)$ and $\O(v)$ are, respectively,
\[ [\O_n (v_n)]_i  =  \langle v_n,  \mathds{1}_{B_\delta( \x_i) \cap \M_n}   \rangle_{L^2(\gamma_n)},\] 
\[ [\O (v)]_i  =  \langle v ,  \mathds{1}_{B_\delta( \x_i) \cap \M} \rangle_{L^2(\gamma)}.\] 
Since
\[  \mathds{1}_{B(\x_i, \delta)\cap \M_n} \converges{TL^2} \mathds{1}_{B(\x_i, \delta) \cap \M},\]
we may use the fact that $v_n \converges{TL^2 } v $ and the continuity of inner products under $TL^2$-convergence (see Proposition 2.6 in \cite{trillosACHA}) to deduce that
\[ [ \O_n (v_n)  ]_i \rightarrow [\O (v) ]_i, \quad \forall i=1, \dots, p .\]
\end{proof}

\subsubsection{Proof of Theorem \ref{main:thm}}
By Proposition \ref{GammaConv} it is enough to establish the $\Gamma$-convergence of $J_n$ towards $J$. Notice that the graph-posterior distributions $\boldsymbol{\mu_n}$ and $\boldsymbol{\mu}$ are the unique minimizers of $J_n$ and $J$, respectively.

\textbf{Liminf inequality:} Suppose that $\boldsymbol{\nu_n} \converges{\mathcal{P}(TL^2)} \boldsymbol{\nu}$ where, as usual, $\boldsymbol{\nu_n} \in \P(L^2(\gamma_n))$, $\boldsymbol{\nu} \in \P(L^2(\gamma))$, and we interpret the statement $\boldsymbol{\nu_n} \converges{\mathcal{P}(TL^2)} \boldsymbol{\nu}$ as
$\mathcal{I}_{\gamma_n \sharp}\boldsymbol{\nu_n} \converges{\mathcal{P}(TL^2)} \mathcal{I}_{\gamma \sharp } \boldsymbol{\nu}.$ First we show that
\begin{equation}
 \liminf_{n \rightarrow \infty} \int_{L^2(\gamma_n)} \phi( \O_n \circ \F_n (u_n) ; y  )d \boldsymbol{\nu_n}(u_n) \geq  \int_{L^2(\gamma)} \phi( \O \circ \F (u) ; y  )d \boldsymbol{\nu}(u).
\label{AuxLiminf}
\end{equation}
For that purpose consider the functions $H_n : TL^2 \rightarrow \R\cup \{ \infty\}$ and $H: TL^2 \rightarrow \R \cup \{ \infty \}$ defined by
\[   H_n\bigl( (\theta, v  ) \bigr) := \begin{cases}  \phi\bigl( \O_n \circ \F_n (  v ) ; y  \bigr),  & \text{ if } \theta=\gamma_n, \\ \infty, & \text{otherwise},    \end{cases}  \]
\[   H\bigl( (\theta, v  ) \bigr) := \begin{cases}  \phi\bigl( \O \circ \F (  v ) ; y  \bigr),  & \text{ if } \theta=\gamma, \\ \infty, & \text{otherwise}.    \end{cases}  \] 
We claim that if  $ \bigl\{ (\theta_n , v_n ) \bigr\}_{n \in \N}$ converges in $TL^2$ towards $(\theta, v)$, then
\[ \liminf_{n \rightarrow \infty} H_n\bigl((\theta_n , v_n)\bigr) \geq H\bigl((\theta, v)\bigr) .   \]
Indeed, by the definition of $H_n$, we can assume without loss of genarality that for all $n \in \N$, $ \theta_n = \gamma_n $; in particular $\theta = \gamma$. The convergence $v_n \converges{TL^2} v$ and the second part of Theorem \ref{lemma} guarantee that $\F_n(v_n ) \converges{TL^2} \F(v)$. In turn, $\F_n(v_n) \converges{TL^2} \F(v)$ implies that $\O_n \circ \F_n(v_n) \converges{TL^2} \O\circ \F (v) $ by the last part of Theorem \ref{lemma}. Finally, the assumed continuity of $\phi(\cdot ; y)$ implies that
\[ \phi\bigl(\O_n \circ \F_n (v_n ) ; y  \bigr) \rightarrow \phi\bigl(\O \circ \F (v) ; y\bigr) ,\]
proving in this way our claim.

We may now apply Lemma \ref{lemmaaux1} to deduce that
\begin{align*}
\begin{split}
\liminf_{n \rightarrow \infty} \int_{L^2(\gamma_n)} \phi\bigl(\O_n \circ \F_n(u_n); y\bigr) \, d \boldsymbol{\nu_n}(u_n) &=  \liminf_{n \rightarrow \infty} \int_{ TL^2} H_n\bigl((\theta, v)\bigr) \, d \I_{\gamma_n \sharp}    \boldsymbol{\nu_n}   \bigl((\theta, v)\bigr) 
\\ & \geq   \int_{ TL^2} H\bigl((\theta, v)\bigr) \, d \I_{\gamma \sharp}   \boldsymbol{\nu}   \bigl((\theta, v)\bigr) 
\\ & = \int_{L^2(\gamma)} \phi\bigl(\O\circ \F(u);y \bigr)\, d\boldsymbol{\nu}(u),
\end{split}
\end{align*} 
establishing \eqref{AuxLiminf}.

On the other hand, the lower semicontinuity of $\dkl(\cdot \| \cdot)$ with respect to weak convergence of probability measures \cite{dupuis2011weak}, together with the first part of Theorem \ref{lemma}, implies that
\begin{align*}
\begin{split}
\liminf_{n \rightarrow \infty} \dkl(\boldsymbol{\nu_n} \| \boldsymbol{\pi_n}) & = \liminf_{n \rightarrow \infty} \dkl( \I_{\gamma_n \sharp} \boldsymbol{\nu_n} \| \I_{\gamma_n \sharp} \boldsymbol{\pi_n} ) 
\\ & \geq \dkl( \I_{ \gamma \sharp} \boldsymbol{\nu} \| \I_{\gamma \sharp} \boldsymbol{\pi}  )
\\ & = \dkl( \boldsymbol{\nu} \| \boldsymbol{\pi} ).
\end{split}
\end{align*}
Combining the above inequality with \eqref{AuxLiminf} we obtain the desired liminf inequality:
\[   \liminf_{n \rightarrow \infty} J_n(\boldsymbol{\nu_n}) \geq J(\boldsymbol{\nu}).  \]

\textbf{Compactness:} Let $\{ \boldsymbol{\nu_n} \}_{n \in \N}$ be a sequence satisfying 
\[ \sup_{n \in \N} J_n(\boldsymbol{\nu_n}) < \infty.\]
In particular, from the fact that $\phi$ is bounded from below it follows that
\[ \sup_{n \in \N} \dkl(\boldsymbol{\nu_n} \| \boldsymbol{\pi_n}) < \infty. \]
We can make use of the stability of priors from Theorem \ref{lemma} and of Lemma \ref{lemmaaux} with $\mathcal{X}= \overline{TL^2}$ to conclude that $\{ \boldsymbol{\nu_n} \}_{n \in \N}$ is pre-compact in $\mathcal{P}( \overline{TL^2})$. To prove that cluster points of $\{\boldsymbol{\nu_n} \}_{n \in \N}$ are actually in $\mathcal{P}(TL^2)$, we simply notice that Lemma \ref{lemma} also implies that cluster points of $\{ \boldsymbol{\nu_n}\}_{n \in \N}$ are absolutely continuous with respect to $\boldsymbol{\pi}$ (a measure that is concentrated in $TL^2$; in fact in $L^2(\gamma)$). Hence, $\{ \boldsymbol{\nu_n} \}_{n \in \N}$ is pre-compact in $\mathcal{P}(TL^2)$.

As stated earlier, the shown $\Gamma$-convergence of $J_n$ towards $J$ guarantees the convergence of posteriors
\[ \boldsymbol{\mu_n} \converges{\mathcal{P}(TL^2) } \boldsymbol{\mu}.\] 
From the above convergence, the stability of forward maps in Theorem \ref{lemma}, and Proposition \ref{compos:prop}, it follows that 
\[ \F_{n \sharp}\boldsymbol{\mu_n} \converges{\mathcal{P}(TL^2)}{  \F_{\sharp} \boldsymbol{\mu}}. \]

\textbf{Limsup inequality:} For the limsup inequality let $\boldsymbol{\nu} \in \mathcal{P}(L^2(\gamma))$ that without loss of generality we can assume satisfies $J(\boldsymbol{\nu}) < \infty$ so that in particular $\dkl(\boldsymbol{\nu}\| \boldsymbol{\pi})< \infty$.  Moreover, we can further assume that
\[ \int_{L^2(\gamma)} \lVert u \rVert^2_{L^2(\gamma )}d \boldsymbol{\nu}(u) < \infty  . \]
Indeed, this is possible due to Remark \ref{DenseSet} and the fact that the set of measures in $\mathcal{P}(L^2(\gamma))$ with finite second moments is dense in $\mathcal{P}(L^2(\gamma))$ with respect to the energy $J$; to see this it is enough to truncate densities, use the monotone convergence theorem, and use the fact that $\boldsymbol{\pi}$ has finite second moment.

We consider the maps $P_n : L^2(\gamma) \rightarrow L^2(\gamma_n)$ defined as 
\[ u= \sum_{i=1}^\infty  a_i \psi_i \longmapsto P_n(u) :=\sum_{i=1}^n \left(\frac{\alpha+ \lambda_i}{\alpha + \lambda_i^n}\right)^{s/4} a_i \psi_i^n. \]

Let $\boldsymbol{\nu_n}:= P_{n \sharp} \boldsymbol{\nu}$ and notice that $\boldsymbol{\pi_n} = P_{n \sharp} \boldsymbol{\pi}$. It follows that
\[ \dkl(\boldsymbol{\nu_n} \| \boldsymbol{\pi_n}) = \dkl( P_{n \sharp} \boldsymbol{\nu} \| P_{n \sharp} \boldsymbol{\pi}  ) \leq \dkl(\boldsymbol{\nu} \| \boldsymbol{\pi}) < \infty, \quad \forall n \in \N.\]
As in the proof of the compactness property we conclude that $\{ \boldsymbol{\nu_n} \}_{n \in \N}$ is pre-compact in $\mathcal{P}(TL^2)$ and moreover all its cluster points are concentrated in $L^2(\gamma)$. We claim that $\{\boldsymbol{\nu_n} \}_{n \in \N}$ converges to $\boldsymbol{\nu}$ by showing that $\boldsymbol{\nu}$ is the only cluster point of $\{\boldsymbol{\nu_n} \}_{n \in \N}$. For that purpose suppose without loss of generality that $\boldsymbol{\nu_n} \converges{\mathcal{P}(TL^2)} \tilde{\boldsymbol{\nu}}$; we want to show that $\tilde{\boldsymbol{\nu}}= \boldsymbol{\nu}$.

By Skorohod's theorem we can find a probability space $(\widetilde{\Omega}, \widetilde{\mathfrak{F}}, \widetilde{P})$ where $TL^2$-random variables $\{ \widetilde{X}_n \}_{n \in \N}$, $\tilde{X}$ can be defined in such a way that
\[ \widetilde{X}_n \sim \boldsymbol{\nu_n} , \forall n \in \N  \text{ and }\quad \tilde{X} \sim \tilde{\boldsymbol{\nu}},\]
and
\[ \widetilde{X}_n(\tilde{\omega}) \rightarrow \tilde{X}(\tilde{\omega}), \quad \widetilde{\mathbb{P}}-a.e. \,\,\tilde{\omega}. \]
Let now $X$ be a $L^2(\gamma)$-random variable with $X \sim \boldsymbol{\nu}$; this random variable may be defined in a probability space different from $(\widetilde{\Omega}, \widetilde{\mathfrak{F}}, \widetilde{P})$.

Fix $k \in \N$. By definition of $\boldsymbol{\nu_n}$ it follows that the random vectors $\vec{V}_n, \vec{\widetilde{V}}_n \in \R^k$, with coordinates
\[\vec{V}_{n,i}:= \left( \frac{\alpha + \lambda_i}{\alpha + \lambda_i^n} \right)^{s/4} \langle X , \psi_i \rangle_{L^2(\gamma)} , \quad i=1, \dots, k, \]
\[\vec{\widetilde{V}}_{n,i}:= \langle \widetilde{X}_n , \psi_i^n \rangle_{L^2(\gamma_n)} , \quad i=1, \dots, k, \]
have the same distributions. Taking limits as $n \rightarrow \infty$, using \eqref{ConvEigenvalue}, \eqref{ConvEigenfunc}, and the continuity of inner products with respect to $TL^2$ convergence, we conclude that the random vectors $\vec{V}, \vec{\widetilde{V}} \in \R^k$, with coordinates
\[\vec{V}_{i}:= \langle X , \psi_i \rangle_{L^2(\gamma)} , \quad i=1, \dots, k, \]
\[\vec{\widetilde{V}}_{i}:= \langle \widetilde{X} , \psi_i \rangle_{L^2(\gamma)} , \quad i=1, \dots, k, \]
have the same distribution. Since finite dimensional projections characterize distibutions of $L^2(\gamma)$-random variables, we deduce that $X$ and $\widetilde{X}$ have the same distribution and so $\boldsymbol{\nu} = \boldsymbol{\tilde{\nu}}$. Thus
\[ \boldsymbol{\nu_n} \converges{\mathcal{P}(TL^2)} \boldsymbol{\nu}, \]
and
\[ \limsup_{n \rightarrow \infty} \dkl(\boldsymbol{\nu_n} \| \boldsymbol{\pi_n}) \leq \dkl(\boldsymbol{\nu} \| \boldsymbol{\pi}).  \]

In the remainder of the proof we show that 
\begin{equation}
 \lim_{n \rightarrow \infty} \int_{L^2(\gamma_n)} \phi( \O_n \circ \F_n( u_n ); y ) d \boldsymbol{\nu_n}(u_n) = \int_{L^2(\gamma)} \phi( \O \circ \F (u); y ) d \boldsymbol{\nu}(u). 
\label{AuxLimsup1}
\end{equation}
First, observe that since $\boldsymbol{\nu_n}= P_{n \sharp} \boldsymbol{\nu}$, we can use the change of variables formula to write
\[ \int_{L^2(\gamma_n)} \phi\bigl( \O_n \circ \F_n( u_n ); y \bigr) d \boldsymbol{\nu_n}(u_n)  = \int_{L^2(\gamma)} \phi\bigl( \O_n \circ \F_n( P_nu ); y \bigr) d \boldsymbol{\nu}(u).  \]
We make use of the dominated convergence theorem to establish \ref{AuxLimsup1}.

Let $u \in L^2(\gamma)$. Then,
\[ \mathcal{F}_n(P_n u) = \sum_{i=1}^n \left(\frac{\alpha + \lambda_i}{\alpha + \lambda_i^n}\right)^{s/4} \langle u , \psi_i \rangle e^{-\lambda_i^n} \psi_i^n.\]
We consider the truncated version of $\mathcal{F}_n(P_n u)$ 
\[    \mathcal{F}_n(P_n u)_{k_n}:= \sum_{i=1}^{k_n} \left(\frac{\alpha + \lambda_i}{\alpha + \lambda_i^n}\right)^{s/4} \langle u , \psi_i \rangle e^{-\lambda_i^n} \psi_i^n.\] 
where $k_n =  \lceil \frac{c}{\veps_n^m} \rceil$ for some small enough constant $c>0$.
Using \eqref{ConvRates}, \eqref{ConvEigenvalue} and \eqref{ConvEigenfunc} it is straightforward to show that
\[ \F_n(P_n u)_{k_n} \converges{TL^2} \F(u).  \] 
On the other hand, 
\begin{align}
\begin{split}
\label{AuxLimsup3}
\lVert \F(P_n u) - \F(P_n u)_{k_n}  \rVert^2_{L^2(\gamma_n)}   & \leq \sum_{i=k_n + 1}^n \left(\frac{\alpha + \lambda_i}{\alpha + \lambda_i^n}\right)^{s/2}\langle  u , \psi_i \rangle^2_{L^2(\gamma)} e^{-2 \lambda_i^n}
\\ & \leq C e^{- 2 \lambda_{k_n}^n}n^{1+s/2m}\lVert u \rVert^2_{L^2(\gamma)}   
\\ & \leq  C \lVert u \rVert_{L^2(\gamma)}^2,
\end{split}
\end{align}
where $C$ is a large enough constant. Since the second to last term in \eqref{AuxLimsup3} converges to zero as $n \rightarrow \infty$, we deduce that
\[ \F_n(P_n u) \converges{TL^2} \F(u), \]
and hence
\[ \lim_{n \rightarrow \infty}\phi\bigl(\O_n \circ \F_n (P_n u) ; y \bigr) = \phi \bigl(\O \circ \F (u) ; y\bigr).  \]
To finish the proof we notice that it is possible to find a large enough constant $C>0$ such that, for all $n \in \N$ and all $u \in L^2(\gamma),$ 
\begin{align*}
\begin{split}
 \phi(\O_n \circ \F_n (P_n u) ; y ) & \leq C \lVert \F_n (P_n u) \rVert^2_{L^2(\gamma_n)} 
\\& = C \lVert \F_n (P_n u)_{k_n} \rVert_{L^2(\gamma_n)}^2 + C \lVert\F_n (P_n u)_{k_n} - \F_n (P_n u)  \rVert_{L^2(\gamma_n)}^2
\\& \leq C \lVert u \rVert^2. 
\end{split}
\end{align*}
Indeed, to obtain the last inequality we have used \eqref{AuxLimsup3} to control the term  $\lVert\F_n (P_n u)_{k_n} - \F_n (P_n u)  \rVert_{L^2(\gamma_n)}^2$; the term  
$\lVert\F_n (P_n u)_{k_n}  \rVert_{L^2(\gamma_n)}^2$ is controlled using the fact that the terms
\[ \frac{\alpha + \lambda_i}{\alpha + \lambda_i^n}, \quad i=1, \dots, k_n \]
are uniformly controlled due to \eqref{ConvRates}. Since the function $u \mapsto \lVert u \rVert^2_{L^2(\gamma)}$ is integrable with respect to $\boldsymbol{\nu}$, the dominated convergence theorem can be applied and \eqref{AuxLimsup1} now follows. 

\qed

  \bibliographystyle{plain}
\bibliography{isbib}

\end{document}